\documentclass[a4paper,12pt,reqno]{amsart}     
\usepackage{amssymb,amsthm,amsmath}  
\usepackage[14pt]{extsizes}  
\usepackage[english]{babel}   
\usepackage{fullpage}    
\usepackage{hyperref}  

\title{Noncommutative Euclidean spaces}
\author{Michel Dubois-Violette, Giovanni Landi}
\address[]{\textit{Michel Dubois-Violette} 
\newline \indent  
Laboratoire de Physique Th\'eorique, UMR 8627, Universit\'e Paris XI,
\newline \indent
B\^atiment 210, F-91 405 Orsay Cedex}
\email{Michel.Dubois-Violette@u-psud.fr}
\address[]{\textit{Giovanni Landi} 
\newline \indent
Matematica,
Universit\`a di Trieste, 
\newline \indent
Via A. Valerio, 12/1, 34127  Trieste, Italy 
\newline \indent
and INFN, Trieste, Italy}
\email{landi@units.it}

\numberwithin{equation}{section}
\newtheorem{theo}{Theorem}[section]
\newtheorem{lemm}[theo]{Lemma}
\newtheorem{prop}[theo]{Proposition}
\newtheorem{defi}[theo]{Definition}

\newtheorem{coro}[theo]{Corollary}

\newcommand{\dd}{\mathrm{d}}
\newcommand{\nn}{\nonumber}

\newcommand{\ca}{\mathcal{A}}
\newcommand{\cc}{\mathcal{C}}
\newcommand{\cf}{\mathcal{F}}

\newcommand{\ccR}{\mathcal{R}}
\newcommand{\fraca}{{\mathfrak A}}
\newcommand{\fracg}{{\mathfrak g}} 
\newcommand{\IC}{{\mathbb C}}   
\newcommand{\IH}{{\mathbb H}}  
\newcommand{\IN}{{\mathbb N}}   
\newcommand{\IR}{{\mathbb R}}   
   
\newcommand{\IS}{{\mathbb S}}   
   
\newcommand{\IZ}{{\mathbb Z}}   

\newcommand{\Rt}{{\IR^4_\theta}} 
\newcommand{\Ct}{{\IC^2_\theta}} 
\newcommand{\car}{{\ca_R}}

\DeclareMathOperator{\Hom}{Hom}
\DeclareMathOperator{\Ext}{Ext}

\DeclareMathOperator{\SO}{SO}        
\DeclareMathOperator{\SU}{SU}

\def\gr{\mathrm{gr}}

\def\dim{{\mbox{dim}}}

\def\Hom {{\mbox{Hom}}}

\def\Lie{{\mbox{Lie}}}

\def\fraca{{\mathfrak A}}
 
 \def\fracg{{\mathfrak g}}

\def\bbbone{\mbox{\rm 1\hspace {-.6em} l}}

\newcommand{\beqa}{\begin{align}}
\newcommand{\eeqa}{\end{align}}
\newcommand{\beq}{\begin{equation}}
\newcommand{\eeq}{\end{equation}}

\begin{document}

%

\begin{abstract}

We give a definition of noncommutative finite-dimensional Euclidean spaces $\mathbb R^n$. We then remind our definition of noncommutative products of Euclidean spaces $\mathbb R^{N_1}$ and $\mathbb R^{N_2}$ which produces noncommutative Euclidean spaces $\mathbb R^{N_1+N_2}$. We solve completely the conditions defining the noncommutative products of the Euclidean spaces $\mathbb R^{N_1}$ and $\mathbb R^{N_2}$ and prove that the corresponding noncommutative unit spheres $S^{N_1+N_2-1}$ are noncommutative spherical manifolds. We then apply these concepts to define ``noncommutative" quaternionic planes  and noncommutative quaternionic tori on which acts the classical quaternionic torus $T^2_{\mathbb H}=U_1(\mathbb H)\times U_1(\mathbb H)$.

\bigskip\bigskip
\centerline{\emph{Dedicated to Alain Connes for his 70th birthday}}
\smallskip
\centerline{(i.e. dix fois ``l'\^age de raison" !)}
\bigskip\bigskip\bigskip\bigskip
\end{abstract}

\maketitle

\tableofcontents
\parskip = 1 ex

\thispagestyle{empty}

\section{Introduction }

In reference \cite{mdv-lan:2017} we have defined ``noncommutative products" of finite-dimensional Euclidean spaces. Our aim here is firstly to complete the description of the properties of these constructions and secondly to put them in the general frame of the theory of regular algebras, that is of the noncommutative generalizations of the algebras of polynomials. Concerning the first point, we shall in particular prove that the  ``noncommutative spheres" underlying the  noncommutative products of Euclidean spaces are noncommutative spherical manifolds in the sense of our joint work with Alain Connes \cite{ac-lan:2001}, 
\cite{ac-mdv:2002a}. Concerning the second point we notice that the study of noncommutative generalizations of algebras of polynomials is the first step to develop a noncommutative algebraic geometry
 and that these algebras play the role of algebras of ``polynomials functions" on ``noncommutative vector spaces" or of algebras of homogeneous coordinates of noncommutative projective spaces.
 
 The plan of the paper is the following one. After this introduction, Section 2 contains a general frame for the noncommutative Euclidean spaces $\mathbb R^n$ . We define and study there the corresponding generalized Clifford algebras. Section 3 is devoted to our definition of noncommutative products  of the Euclidean space $\mathbb R^{N_1}$ with the Euclidean space $\mathbb R^{N_2}$ and to the description of its first properties. In Section 4 we completely solve the conditions defining  noncommutative products of the Euclidean spaces $\mathbb R^{N_1}$ and $\mathbb R^{N_2}$ and show that the corresponding noncommutative unit spheres $S^{N_1+N_2-1}$ are  noncommutative spherical manifolds in the sense of \cite{ac-lan:2001}, \cite{ac-mdv:2002a}. Finally in Section 5 we use the above constructions to define noncommutative 2-dimensional quaternionic spaces and noncommutative quaternionic tori.
 
For sake of completeness    we have included at the end of the paper an appendix concerning the homogeneous and nonhomogeneous quadratic algebras.\\
  
 \noindent \underbar{Notations}. We use throughout the Einstein convention of summation over repeated up-down indices. All our algebras are associative unital (mostly) complex $\ast$-algebras. Without other mention a graded algebra means here a $\mathbb N$-graded algebra.
 
  \section{Noncommutative Euclidean spaces}
 \subsection{Generalities} The correspondence (or duality) between spaces and algebras of functions on these spaces is by now familiar. The idea of noncommutative geometry is to forget the commutativity of the algebras of functions and to replace them by appropriate classes of noncommutative associative algebras and consider that these are ``algebras of functions" on some (fictive) ``noncommutative spaces".
 
 For instance, the natural algebras of functions on finite-dimensional vector spaces are the algebras of polynomial functions generated by the linear forms i.e. by the coordinates. In these polynomial algebras the coordinates commute. Given a class of noncommutative associative algebras generalizing the polynomial algebras, one may consider algebras generated by coordinates in which they satisfy other relations  than the commutation between them and defining thereby noncommutative vector spaces.
 
 Concerning the classes of noncommutative algebras generalizing the polynomial algebras, there are several possible choices. A very natural wide choice is the class of Artin-Schelter algebras \cite{art-sch:1987}, a wider choice consists in forgetting the polynomial growth condition. One can also restrict attention to quadratic algebras since algebras of polynomials are quadratic. Our constructions will be in the class of quadratic Artin-Schelter algebras which are Koszul. The usual polynomial algebras are the commutative algebras in this class as in the other above classes. We refer to algebras of the appropriate class as {\sl regular algebras}.
 
 \subsection{Noncommutative real vector spaces}
 
 There is an important subtlety hidden behind the above discussion which concerns  the reality conditions. Consider the space $\mathbb R^n$ with its canonical basis of coordinates $x^k, k\in \{1,\dots, n\}$. Then there are two natural algebras of polynomial functions on $\mathbb R^n$ : The real algebra of real polynomials functions on $\mathbb R^n$; or the complex $\ast$-algebra of complex polynomial functions on $\mathbb R^n$. Since the first algebra is the real subalgebra of hermitian elements of the second algebra this makes practically no difference. However in the noncommutative case  the situation is different. Should one consider that a noncommutative $\mathbb R^n$ corresponds to a regular real algebra generated by the coordinates $x^k$ or should one consider that a noncommutative $\mathbb R^n$ corresponds to a regular complex $\ast$-algebra generated by the hermitian elements $x^k, k\in\{1,\dots,n\}$? There is a big difference now because the real subspace of the hermitian elements of a noncommutative complex $\ast$-algebra is not an associative algebra, (it is a real Jordan algebra for the symmetrized product). Anyone knowing a little the quantum theory would say that the second choice is the right one. This is our choice in the following, that is a noncommutative $\mathbb R^n$ correspond to a regular complex $\ast$-algebra generated by the hermitian elements $x^k,k\in \{1,\dots,n\}$.
 
 In all cases regular algebras are always connected graded algebras finitely generated in degree 1 with a finite presentation. For that kind of algebra and more generally for the class of graded connected algebras,  all the homological dimensions coincide : the projective dimension of the ground field as trivial module coincides with the global dimension \cite{car:1958} and coincides with the Hochschild dimension in homology as in cohomology \cite{ber:2005}. We shall refer to it as the {\sl global dimension}.
 
 \subsection{Noncommutative Euclidean spaces}
 
 Let $\mathbb R^n$ be now the Euclidean $\mathbb R^n$ with metric $\sum_kx^k\otimes x^k$ where the coordinates $x^k$ are the dual basis of the canonical orthonormal basis of $\mathbb R^n$. A noncommutative Euclidean space $\mathbb R^n$ will be defined as corresponding to a regular complex $\ast$-algebra $\ca$ generated by the hermitian elements $x^k$ such that the quadratic element $\sum_k(x^k)^2$ belongs to the center $Z(\ca)$ of $\ca$.
 
 An important class of such algebras appeared in our joint work with Alain Connes \cite{ac-lan:2001}, \cite{ac-mdv:2002a}. Indeed in this work a large class of noncommutative spheres $S^{n-1}$, refered to as noncommutative spherical manifolds, are defined by $K$-homological conditions (satisfied by the classical spheres). The homogeneisation of these conditions leads to quadratic $\ast$-algebras $\ca$ generated by $n$ hermitian elements $x^k,k\in \{1,\dots, n\}$, such that $\sum_k(x^k)^2$ is central and then the noncommutative spheres $S^{n-1}$ appear as the unit spheres in the noncommutative Euclidean spaces $\mathbb R^n$ corresponding to the $\ast$-algebras $\ca$, i.e. the noncommutative $n-1$ spheres correspond to the quotient algebras of the $\ca$ by the ideals generated by $\sum(x^k)^2-1$. The algebras $\ca$ for $n\leq 4$ have been analyzed in details in \cite{ac-mdv:2003} and \cite{ac-mdv:2008}. It turns out in particular that they are Koszul Artin-Schelter algebras of global dimensions $n$ and are furthermore Calabi-Yau \cite{gin:2006} algebras since the preregular multilinear forms corresponding to the ``volume forms" are graded-cyclic potentials (or superpotentials) for these algebras \cite{mdv:2007}. It is expected that these properties still hold for the $\ca$ with arbitrary $n\in \mathbb N$. For each $n$, there is a family of algebras $\ca$, as above which contains of course the classical commutative one, that is the algebra $\mathbb C[x^k]$ of complex polynomials functions on the Euclidean space $\mathbb R^n$ (which is the only commutative algebra of the family).
 
 In the following {\sl a noncommutative Euclidean space} $\mathbb R^n$ will dualy be defined by a quadratic complex $\ast$-algebra $\ca$ generated by $n$ hermitian elements $x^k(k\in\{1,\dots, n\}$) which is Koszul and Artin-Schelter of global dimension $n$ and which is such that $\sum^n_{k=1}(x^k)^2$ is central in $\ca$, (i.e. is an element  of the center $Z(\ca)$ of $\ca$). In this context, $\ca$ is the analog of the algebra of complex polynomial functions on $\mathbb R^n$, (i.e. of the symmetric algebra of the complexified space of the dual of $\mathbb R^n$), while the Koszul dual $\ca^!$ of $\ca$ is the analog of the exterior algebra of the complexified space of $\mathbb R^n$. Since $\ca$ is Koszul it follows that $\ca^!$ is also Koszul and since $\ca$ is Gorenstein it follows that $\ca^!$ is a graded Frobenius algebra \cite{smi:1996}, \cite{lu-pal-wu-zha:2004}. One has $\ca^!=\oplus^n_{k=0}\ca^!_k$ where $\dim(\ca^!_k)=\dim(\ca^!_{n-k})$ so in particular $\dim(\ca^!_n)=1$ and any element $w\not= 0$ of the dual $(\ca^!_n)^\ast$ of $\ca^!_n$ is a preregular multilinear form which is a potential (or superpotential) for $\ca$, 
 \cite{mdv:2007}.
 
 \subsection{Generalized Clifford algebras}
 
 Let $\ca$ be a quadratic algebra and let $\ca^!$ be its Koszul dual. Let $(x^k)$ be a basis of $\ca_1$ with dual basis $(\theta_k)$ of $\ca^!_1$ ($k \in \{1,\dots,n\}$),  where $\ca_1$ and $\ca^!_1$ are the subspaces of the elements of degree 1 of $\ca$ and $\ca^!$. One has the following lemma.
 \begin{lemm}\label{eq}
 In the algebra $\ca^!\otimes \ca$ one has $(\theta_k x^k)^2=0$ and moreover given the relations of the $\theta_k$ in $\ca^!$ the relations of $\ca$ are equivalent to $(\theta_kx^k)^2=0$ for the $x^k$.
 \end{lemm}
 The proof of this lemma follows from the definition of the Koszul duality. Notice also that one can exchange the roles of $\ca$ and $\ca^!$ by using $(\ca^!)^!=\ca$. 
 
 Let now $\ca$ be a quadratic $\ast$-algebra corresponding to a noncommutative Euclidean space $\mathbb R^n$ with orthonormal coordinates $x^k$ ($k\in \{ 1,\dots, n\}$) as in \S 2.3 and let $(\theta_k)$ be the dual basis of $\ca^!_1$. Let 
\[
 r^{k\ell}_\Lambda \theta_k \theta_\ell=0,\,  (\Lambda\in \{1,\dots,m\})
 \]
  be a basis of the quadratic relations of $\ca^!$ which is hermitian in the sense that one has
 \begin{equation}
 \overline{r^{k\ell}_\Lambda}=r^{\ell k}_\Lambda
 \label{sab}
 \end{equation}
 for any $\Lambda\in \{1,\dots, m\}$ and $k,\ell\in\{1,\dots,n\}$.
 
 The following definition of generalized Clifford algebra which extends the one used in \cite{ac-mdv:2002a} in particular cases (e.g. for $\theta$-deformations) is taken  from \cite{bel-ac-mdv:2011} which is still in preparation.

 \begin{defi}\label{Cl}
 The (generalized) Clifford algebra $\cc\ell(\ca)$ of $\ca$ is the complex unital $\ast$-algebra generated by $n$ hermitian elements $\Gamma_k$ with relations
 \begin{equation}\label{Gama}
 r_\Lambda^{k\ell} \Gamma_k \Gamma_\ell=\sum^n_{i=1} r_\Lambda^{ii}\bbbone
 \end{equation}
 for any $\Lambda\in \{1,\dots,m\}$.
 \end{defi}
 The algebra $\cc\ell(\ca)$ is a non homogeneous quadratic algebra with $\ca^!$ as  quadratic part and only non homogeneous terms of degree 0. It is not $\mathbb N$-graded but only $\mathbb Z_2$-graded and filtered by
 \[
 \cf^q = F^q(\cc\ell(\ca)) = \ \{ \text{elements of degree in}\  \Gamma\leq q \}
 \]
 for  $q\in \mathbb N$ and  $\cf^q=0$ for $q\leq -1$. One has a canonical surjective homomorphism of graded algebras of $\ca^!$  onto the associated graded algebra $\gr(\cc\ell(\ca))$
 \begin{equation}\label{can}
 \text{can} : \ca^!\rightarrow \gr(\cc\ell(\ca))=\oplus_p \cf^p/\cf^{p-1}
 \end{equation}
 which induces the isomorphism $\ca^!_1\simeq \cf^1/\cf^0$ of vector spaces.
 
 The  Koszulity of $\ca^!$ and the centrality of $\sum_k(x^k)^2$ in $\ca$ imply the following Poincar\'e-Birkhoff-Witt (PBW) property.
 
 \begin{prop}\label{PBW}
 
 The canonical homomorphism $\mathrm{can}$ of  \eqref{can} is an isomorphism of graded algebras.
 \end{prop}
 
 Thus any basis of $\ca^!$ gives a basis of $\cc\ell(\ca)$, in particular $\dim(\cc\ell(\ca))=\dim(\ca^!)$.
 \begin{coro}\label{Rprime}
 With the above definition and notations the quadratic relations defining $\ca$ are equivalent to
 \begin{equation}
 (\Gamma_kx^k)^2=\bbbone\otimes \sum_p(x^p)^2
 \label{relpr}
 \end{equation}
for the $x^k$.
 \end{coro}
 
 This is an immediate consequence of Lemma \ref{eq} and Proposition \ref{PBW} above. Notice that a basis of the quadratic relations of $\ca$ reads $r^S_{k\ell} x^kx^\ell=0$ for $S\in\{1,\dots,n^2-m\}$ where the $r^S_{k\ell}$ which can be chosen hermitian satisfy the orthogonality relations 
 \begin{equation}
 r^{k\ell}_\Lambda r^S_{k\ell}=0
 \label{orth}
 \end{equation}
 for $\Lambda\in \{1,\dots,m\}, S\in\{1,\dots,n^2-m\}$.
 
 \section{Noncommutative products of Euclidean spaces}
 
 In this section we recall the definition introduced in \cite{mdv-lan:2017} of the noncommutative products of the Euclidean space $\mathbb R^{N_1}$ with the Euclidean space $\mathbb R^{N_2}$. The letters $\lambda, \mu, \nu, \rho, \dots$ from the middle of the Greek alphabet run in $\{1,\dots,N_1\}$ and label the coordinates of $\mathbb R^{N_1}$ while the letters $\alpha, \beta, \gamma, \delta,\dots$ from the beginning of the Greek alphabet run in $\{1,\dots,N_2\}$ and label the coordinates of $\mathbb R^{N_2}$.
 
 \subsection{The quadratic $\ast$-algebra $\ca_R$}
 Let $\ca_R$ be the complex quadratic  algebra generated by the independent elements $x^\lambda_1$ and $x^\alpha_2$ ($\lambda\in \{1,\dots, N_1\},$  $\alpha\in\{1,\dots, N_2\})$ with relations 
 \begin{equation}
 x^\lambda_1 x^\mu_1=x^\mu_1 x^\lambda_1,\>\>\> x^\alpha_2x^\beta_2=x^\beta_2x^\alpha_2,\> \> \> x^\lambda_1x^\alpha_2=R^{\lambda\alpha}_{\beta\mu}\>\>\> x^\beta_2 x^\mu_1,\>\>\>  x^\alpha_2 x^\lambda_1= \bar R^{\lambda\alpha}_{\beta\mu} x^\mu_1 x^\beta_2
 \label{alde}
 \end{equation}
for a suitable class of ``matrix" $R=(R^{\lambda\alpha}_{\beta\mu})$,  where the $\bar R^{\lambda\alpha}_{\beta\mu} \in \mathbb C$ denote  as usual the complex conjugates of the $R^{\lambda\alpha}_{\beta\mu}\in \mathbb C$. The class of the $R$ will be defined progressively by conditions throughout the whole section 3.

As any quadratic complex algebra, $\ca_R$ is a graded algebra $\ca_R=\oplus_{n\in\mathbb N}(\ca_R)_n$ which is connected, that is $(\ca_R)_0=\mathbb C\bbbone$. 

The quadratic relations (\ref{alde}) of $\ca_R$ imply that there is a unique structure of $\ast$-algebra on $\ca_R$ for which the $x^\lambda_1$ $(\lambda\in \{1,\dots,N_1\})$ and the $x^\alpha_2$ $(\alpha\in \{1,\dots, N_2\})$ are hermitian. In other words there is a unique antilinear (=conjugate-linear) involution $f\mapsto f^\ast$ on $\ca_R$ such that $(fg)^\ast=g^\ast f^\ast$ for $f,g\in\ca_R$ and such that $x^\lambda_1=(x^\lambda_1)^\ast$ and $x^\alpha_2=(x^\alpha_2)^\ast$ for $\lambda\in \{1,\dots, N_1\}$, $\alpha\in\{1,\dots,N_2\}$. Furthermore this structure is graded in the sense that one has $f^\ast \in(\ca_R)_n\Leftrightarrow f\in (\ca_R)_n$. In the following $\ca_R$ is endowed with this $\ast$-algebra structure and we refer to $\ca_R$ as the quadratic $\ast$-algebra generated by the hermitian elements  $x^\lambda_1$ and $x^\alpha_2$ with the relations (\ref{alde}).

The $x^\lambda_1 x^\mu_1$ for $\lambda\leq \mu$ and the $x^\alpha_2 x^\beta_2$ for $\alpha\leq \beta$ are linearly independent in $(\ca_R)_2$ and generate $(\ca_R)_2$ with the $x^\lambda_1 x^\alpha_2$. It is also natural in this context to assume that the $x^\alpha_1 x^\lambda_2$ are independent which implies the equations
\begin{equation}
\bar R^{\lambda\alpha}_{\beta\mu} R^{\mu\beta}_{\gamma\nu}= \delta^\lambda_\nu \delta^\alpha_\gamma
\label{real}
\end{equation}
which then imply that the $x^\lambda_1 x^\alpha_2$ are also independent. Finally this implies in particular that the $x^\lambda_1 x^\mu_1$ with $\lambda\leq \mu$, the $x^\alpha_2 x^\beta_2$ with $\alpha\leq \beta$ and the $x^\nu_1 x^\gamma_2$ define a basis of $(\ca_R)_2$ while by definition the $x^\lambda_1$ and the $x^\alpha_2$ form a basis of $(\ca_R)_1$.

More generally relations (\ref{alde}) imply that the $x^{\lambda_1}\dots x^{\lambda_p} x^{\alpha_1}\dots x^{\alpha_{n-p}}$ for $\lambda_i\leq \lambda_{i+1}$, $\alpha_j\leq \alpha_{j+1}$ and $0\leq p\leq n$ generate $(\ca_R)_n$, but it is natural to expect more, namely that they are linearly independent forming therefore a basis of $(\ca_R)_n$. Indeed since we wish that $\ca_R$ corresponds to a noncommutative product of $\mathbb R^{N_1}$  with  $\mathbb R^{N_2}$ that is that $\ca_R$ is a ``noncommutative tensor product" algebra of the $\ast$-algebra $\mathbb C[x^\lambda_1]$ of the complex polynomials  on $\mathbb R^{N_1}$ with the $\ast$-algebra $\mathbb C [x^\alpha_2]$ of the complex polynomials on $\mathbb R^{N_2}$ it is natural to assume that one has the isomorphism 
\begin{equation}
\ca_R\simeq \mathbb C[x^\lambda_1]\otimes \mathbb C[x^\alpha_2]
\label{isovec}
\end{equation}
of graded complex vector spaces. This will be a consequence of the Yang-Baxter condition that we shall impose in the next subsection 3.2.

Let $(\theta^1_\lambda,\theta^2_\alpha)$ denote the dual basis of the basis $(x^\mu_1,x^\beta_2)$ of $(\ca_R)_1$. Then the Koszul dual algebra of the quadratic algebra $\ca_R$ is the quadratic algebra $\ca^!_R$ generated by the $\theta^1_\lambda$ and $\theta^2_\alpha$ $(\lambda\in\{1,\dots,N_1\}, \alpha\in\{1,\dots,N_2\})$ with relation
\begin{equation}
\theta^1_\lambda\theta^1_\mu=-\theta^1_\mu\theta^1_\lambda,\>\>\> \theta^2_\alpha\theta^2_\beta=-\theta^2_\beta\theta^2_\alpha,\> \>\> \theta^2_\beta\theta^1_\mu=-R^{\lambda\alpha}_{\beta\mu} \theta^1_\lambda\theta^2_\alpha, \> \>\>\theta^1_\mu \theta^2_\beta=-\bar R^{\lambda\alpha}_{\beta\mu} \theta^2_\alpha \theta^1_\lambda
\label{aldek}
\end{equation}
for $\lambda,\mu\in \{1,\dots,N_1\},\ \alpha,\beta\in\{1,\dots,N_2\}$.

It is a consequence of (\ref{alde}) and (\ref{real}) that the $\theta^1_\lambda \theta^1_\mu$ with $\lambda<\mu$, the $\theta^2_\alpha\theta^2_\beta$ with $\alpha<\beta$  and the $\theta^1_\nu\theta^2_\gamma$ define a basis of $(\ca^!_R)_2$ and as above, it is natural to expect that, more generally, one has the isomorphism
\beq\label{disovec}
\ca^!_R\simeq \wedge\mathbb C^{N_1+N_2}
\end{equation}
of graded vector spaces where $\wedge\mathbb C^{N_1+N_2}$ denotes the exterior algebra of the complexified space of $\mathbb R^{N_1+N_2}$. This will be also a consequence of the Yang-Baxter condition that we now define.

\subsection{The involutive matrix $\ccR$ and the Yang-Baxter condition}
Let us put together the coordinates,  defining the
$x^a$ for $a\in \{1,2, \dots, N_1+N_2\}$ by $x^\lambda=x^\lambda_1$ and 
$x^{\alpha+N_1}=x^\alpha_2$. Then, the relations \eqref{alde} can be written in the form
 \beq\label{aldeb}
x^a x^b = \ccR^{a \, b}_{c \, d} \,\, x^c x^d 
\eeq
with the following expressions for  the $\ccR^{a \, b}_{c \, d}$, 
\begin{align}\label{rtot}
 \ccR^{\lambda \mu}_{\nu \rho} & = \delta^\lambda_\rho \, \delta^\mu_\nu ,
\qquad  \ccR^{\gamma \delta}_{\alpha \beta} = \delta^\gamma_\beta \, \delta^\delta_\alpha \nn \\ 
\ccR^{\lambda \alpha}_{\beta \mu} & = R^{\lambda \alpha}_{\beta \mu} ,\qquad 
\ccR^{\alpha \lambda}_{\mu \beta} = \overline{R}^{\lambda \alpha}_{\beta \mu} \nn \\
\ccR^{\lambda \mu}_{\alpha \nu} & = \ccR^{\lambda \mu}_{\alpha \beta} = \ccR^{\lambda \mu}_{\nu \beta} =0 , \nn \\
\ccR^{\alpha \gamma}_{\lambda \beta} & = \ccR^{\alpha \gamma}_{\lambda \mu} = \ccR^{\alpha \gamma}_{\beta \mu} = 0 , \nn \\
\ccR^{\lambda \alpha}_{\mu \nu} & =  \ccR^{\lambda \alpha}_{\beta \gamma} 
= \ccR^{\lambda \alpha}_{\mu \beta} =0 , \nn \\
\ccR^{\alpha \lambda}_{\mu \nu} & = \ccR^{\alpha \lambda}_{\beta \gamma} = \ccR^{\alpha \lambda}_{\beta \mu} = 0 
\end{align}
for $\lambda, \mu, \nu, \rho\in\{1,\dots,N_1\}$ and $\alpha,\beta,\gamma,\delta\in \{1,\dots,N_2\}$.

The $\ccR^{ab}_{cd}$ are the matrix elements of an endomorphism of $(\ca_R)_1\otimes (\ca_R)_1$.
It follows from \eqref{real} that the $\ccR$ matrix in \eqref{rtot} is involutive, that is
\beq\label{involr}
\ccR^2 = I \otimes I
\eeq
where $I$ is the identity mapping of $(\ca_R)_1$ onto itself.

We next  impose that the matrix $\ccR$ satisfies the Yang-Baxter equation
\beq\label{rvan}
(\ccR \otimes I) (I\otimes \ccR) (\ccR \otimes I) = (I \otimes \ccR)(\ccR \otimes I) (I \otimes \ccR) 
\eeq
which in component, for indices $a, b, c,\dots \in\{1,2, \dots, N_1+N_2\}$ reads
\beq\label{rvan1}
\big( (\ccR \otimes I) (I\otimes \ccR) (\ccR \otimes I) - (I \otimes \ccR)(\ccR \otimes I) (I \otimes \ccR) \big)^{a b c}_{\cdots} = 0  
\eeq
And by using the expressions \eqref{rtot} leads to 

\begin{prop}
The cubic (Yang-Baxter) equations \eqref{rvan1} for the matrix $\ccR$ given in \eqref{rtot} are equivalent to the following quadratic conditions for the matrix $R$:
\beq\label{yb1}
\left\{
\begin{array}{ll}
R^{\lambda\alpha}_{\gamma\rho} R^{\rho\beta}_{\delta\mu} 
= R^{\lambda\beta}_{\delta\rho} R^{\rho\alpha}_{\gamma\mu} & \qquad \textup{for indices} \qquad (a, b, c)=(\lambda\alpha\beta)\\
\\
\overline{R}^{\lambda\alpha}_{\gamma\rho} \overline{R}^{\rho\beta}_{\delta\mu} 
= \overline{R}^{\lambda\beta}_{\delta\rho}\overline{R}^{\rho\alpha}_{\gamma\mu} & \qquad \textup{for indices} \qquad (a, b, c)= (\alpha\beta\lambda)\\
\\
\overline{R}^{\lambda\alpha}_{\gamma\rho} R^{\rho\beta}_{\delta\mu} 
= R^{\lambda\beta}_{\delta\rho} \overline{R}^{\rho\alpha}_{\gamma\mu} & \qquad \textup{for indices} \qquad (a, b, c)= (\alpha\lambda\beta)
\end{array}
\right.
\eeq
and
\beq \label{yb2}
\left\{
\begin{array}{ll}
R^{\lambda\alpha}_{\gamma\nu} R^{\mu\gamma}_{\beta\rho} 
= R^{\mu\alpha}_{\gamma\rho} R^{\lambda\gamma}_{\beta\nu} & \qquad \textup{for indices} \qquad (a, b, c)= (\lambda\mu\alpha)\\
\\
\overline{R}^{\lambda\alpha}_{\gamma\nu} \overline{R}^{\mu\gamma}_{\beta\rho} 
= \overline{R}^{\mu\alpha}_{\gamma\rho}\overline{R}^{\lambda\gamma}_{\beta\nu} & \qquad \textup{for indices} \qquad (a, b, c)=(\alpha\lambda\mu)\\
\\
R^{\lambda\alpha}_{\gamma\nu} \overline{R}^{\mu\gamma}_{\beta\rho} 
= \overline{R}^{\mu\alpha}_{\gamma\rho}  R^{\lambda\gamma}_{\beta\nu} & \qquad \textup{for indices} \qquad (a, b, c)= (\lambda\alpha\mu)  
\end{array}
\right.  
\eeq
while they are trivially satisfied for indices $(a,b,c)=(\lambda,\mu,\nu)$ and $(a,b,c)=(\alpha,\beta,\gamma)$ where $\lambda,\mu,\nu\in\{1,\dots,N_1\}$ and $\alpha,\beta,\gamma\in \{1,\dots,N_2\}$.
\end{prop}

\subsection{Noncommutative products of real vector spaces} 
The classical (commutative) solution of conditions \eqref{real}, \eqref{yb1} and \eqref{yb2}, $R=R_0$ is given by
\beq\label{classic}
(R_0)^{\lambda\alpha}_{\beta\mu}=\delta^\lambda_\mu  \delta^\alpha_\beta
\eeq 
and the corresponding algebra $\ca_{R_0}$ is the algebra of complex polynomial functions on the product 
 $\IR^{N_1}\times \IR^{N_2}$. For this reason, we define the \emph{noncommutative product} 
 of $\IR^{N_1}$ and $\IR^{N_2}$  
\beq\label{ncp}
\IR^{N_1}\times_{R} \IR^{N_2}
\eeq
by (duality to) the (coordinate) algebra $\car$ for a general matrix $R$ 
that satisfies the conditions \eqref{real}, \eqref{yb1} and \eqref{yb2}.

\subsection{Regularity} 
From now we assume that the matrix $R$ of relations \eqref{alde} for the algebra $\car$ satisfies conditions \eqref{real}, \eqref{yb1} and \eqref{yb2}. The involutive condition \eqref{involr} entails that we are in the case of a permutation symmetry and thus in the more general context of Hecke symmetries of \cite{gur:1990}  and  \cite{wam:1993}.
\begin{theo}
The algebra $\car$ is Koszul of global dimension $N_1+N_2$ and is Gorenstein with polynomial growth (i.e. is Artin-Schelter).
Moreover,  the Poincar\'e series of $\car$ and of $\car^!$ are given by
\[
P_{\car}(t) = \frac{1}{(1-t)^{N_1+N_2}}  \qquad \mbox{and} \qquad P_{\car^!}(t) = (1+t)^{N_1+N_2} 
\]
thus they are the classical ones.
\end{theo}
This follows from the results of \cite{gur:1990} (see also \cite{wam:1993}) together with the particular structure of Relations \eqref{alde}. Indeed the Yang-Baxter equation implies that there is no ambiguities when one uses \eqref{alde} to convert to the other order the $x^{a_1}\dots x^{a_n}$ for say $a_1\leq a_2\leq\dots\leq a_n$.

Concerning the Koszulity and the Gorenstein property with global dimension $N_1+N_2$, let us sketch briefly the method of \cite{gur:1990} and of \cite{wam:1993}.

We refer to Appendix A for the relevant notions that we need. 
To the involutive matrix  $\ccR$  are associated two quadratic algebras. The first one  is the generalization of the symmetric algebra and is generated by the elements $x^a$ with relations $x^ax^b=\ccR^{ab}_{cd}x^c x^d$ and thus coincides here with $\ca_R$. The second one $\wedge_R$, which is the generalization of the exterior algebra, is generated by the elements $dx^a$ with relations $dx^a dx^b+\ccR^{ab}_{cd} dx^c dx^d=0$ and coincides with the dual $(\ca^!_R)^\ast$ of the Koszul dual $\ca^!_R$ of $\ca_R$ as graded vector space.

On the free graded left $\ca_R$-module
$\ca_R \otimes \wedge^{\bullet}_{R}$
one has two differentials, the Koszul differential 
$$
\delta :  (\ca_{R})_k \otimes \wedge^l_{R} \to (\ca_{R})_{k+1} \otimes \wedge^{l-1}_{R} 
$$
and the generalization of the exterior differential 
$$
\dd :  (\ca_{R} )_k \otimes \wedge^l_{R} \to (\ca_{R})_{k-1} \otimes \wedge^{l+1}_{R} 
$$
which satisfy the classical relation
$$
( \delta  \dd + \dd \delta ) \alpha = (k+l) \alpha \qquad \mbox{for} \quad \alpha \in(\ca_{R} )_k  \otimes \wedge^l_{R} . 
$$

\noindent
This implies both the Koszulity of $\ca_R$ and a generalization of the Poincar\'e lemma. 
Then the elements $(x^{a_1} \cdots x^{a_p})$ for $a_1 \leq a_2 \cdots \leq a_p$ and $p\in \IN$,  
form a homogeneous basis of the graded vector space $\car$ which implies that its Poincar\'e series is given by 
$$
P_{\car}(t) := \sum_{p\in\IN} \dim_{\IC}((\car)_p) \, t^p= \frac{1}{(1-t)^{N_1+N_2}}  
$$
which is the classical one. This is also the case for $\ca^!_R$ where one has
\[
P_{\ca^!_R}(t)=P_{\wedge_{R}}(t)=(1+t)^{N_1+N_2}=(P_{\ca_R}(-t))^{-1}
\]
as expected by the Koszulity property. A homogeneous basis of  $\ca^!_R$ is given by the  $\theta_{a_1}\dots \theta_{a_p}$ with $a_1<a_2<\dots<a_p$ and $p\in \{1,\dots, N_1+N_2\}$.

As a consequence the ground field $\IC$ has projective dimension  $N_1+N_2$ which coincides with the global dimension
$$
\mbox{gl}(\car) = N_1+N_2 ,
$$
of $\car$ and coincides also with its Hochschild dimension \cite{car:1958}, \cite{ber:2005} .

Next, by applying the functor $\Hom_{\car}(\, \cdot  , \car)$ to the Koszul chain complex of free left 
$\car$-modules $(\car \otimes (\car^!)^*_\bullet, \delta)$, one obtains a cochain complex of right 
$\car$-modules
the cohomology of which is $\Ext^\bullet_\car(\IC,\car)$ by definition.  Finally the Gorenstein property
$$
\Ext^n_\car(\IC,\car) = \delta^n_{N_1+N_2} \, \IC 
$$
follows from the above results.
 
\subsection{Noncommutative products of Euclidean spaces}

We now consider the Euclidean spaces 
$\IR^{N_1}$ and $\IR^{N_2}$.

With the generators labelled as before by
$x^a = (x^\lambda=x^\lambda_1, \,  x^{\alpha+N_1}=x^\alpha_2)$
for index $a\in \{1,2, \dots, N_1+N_2\}$ and the relations as in \eqref{aldeb} (or \eqref{alde})
we have the following. 
\begin{theo} 
The following conditions $\mathrm{(i)}$ $\mathrm{(ii)}$ and $\mathrm{(iii)}$ are equivalent:
\begin{align}\label{cerad}
\mathrm{(i)} & \qquad\qquad \sum^{N_1+N_2}_{a=1} (x^a)^2=\sum^{N_1}_{\lambda=1} (x^\lambda_1)^2 + \sum^{N_2}_{\alpha=1} (x^\alpha_2)^2 \qquad  \mbox{is central in} \quad \car,  \nn \\
\mathrm{(ii)} & \qquad\qquad \sum^{N_1}_{\lambda=1}(x^\lambda_1)^2 \quad \mbox{and} \quad \sum^{N_2}_{\alpha=1}(x^\alpha_2)^2
\qquad  \mbox{are in the center of} \quad  \car, \nn \\
\mathrm{(iii)} & \qquad\qquad \sum^{N_1}_{\lambda=1} R^{\lambda\gamma}_{\beta\nu} 
R^{\lambda\beta}_{\alpha \mu}= \delta^\gamma_\alpha \delta_{\mu\nu} \qquad \mbox{and} \qquad  \sum^{N_2}_{\alpha=1} R^{\lambda\alpha}_{\beta\rho} R^{\rho\alpha}_{\gamma\mu}=\delta^\lambda_\mu \delta_{\beta\gamma}. 
\end{align}
\begin{proof}
Since any two components $x_1^\lambda$ and $x_1^\mu$ commute among themselves, and the same is true for any two components $x_2^\alpha$ and $x_2^\beta$, the equivalence of points $\mathrm{(i)} $ and $\mathrm{(ii)} $ follows.
For the rest one just computes using the defining relations \eqref{alde}
$$
\sum_\lambda (x_1^\lambda)^2 x_2^\gamma = \sum_\lambda x_1^\lambda x_1^\lambda x_2^\gamma = 
\sum_\lambda x_1^\lambda R^{\lambda \gamma}_{\beta \nu} x_2^\beta x_1^\nu 
=\sum_\lambda R^{\lambda \gamma}_{\beta \nu} R^{\lambda \beta}_{\alpha \mu}  \,\, 
x_2^\alpha x_1^\mu x_1^\nu .
$$
Asking that the right-hand side be $\sum_\lambda x_2^\gamma (x_1^\lambda)^2$ yields the first condition in point 
$\mathrm{(iii)}$. The second one is obtained along similar lines. 
\end{proof}
\end{theo}

We take the matrix $R$ to satisfies \eqref{cerad} also and define the \emph{noncommutative product of the Euclidean spaces} $\IR^{N_1}$ with $\IR^{N_2}$ 
to be dual of the algebra $\car$ with these additional conditions. 
Clearly, the \eqref{cerad} are satisfied by the classical $R=R_0$. The algebra
$\car$ generalizes the algebra of polynomial functions on the product $\IR^{N_1} \times \IR^{N_2}$. It is clear that so defined, the noncommutative product of the Euclidean spaces $\mathbb R^{N_1}$ and $\mathbb R^{N_2}$ is a noncommutative Euclidean space $\mathbb R^{N_1+N_2}$ in the sense of \S 2.3.
 
The centrality conditions in \eqref{cerad} taken together with the reality conditions \eqref{real} leads to the following conditions on the matrix $R^{\lambda\alpha}_{\beta\mu}$. 
\begin{prop}
Relations \eqref{real} and \eqref{cerad} are equivalent to
\beq\label{eucl0}
R^{\lambda\beta}_{\alpha\mu}=R^{\mu\alpha}_{\beta\lambda}
=\overline{R}^{\mu\beta}_{\alpha\lambda}=(R^{-1})^{\beta\mu}_{\lambda\alpha} 
\eeq
which in turn implies that Relations \eqref{yb1} and \eqref{yb2} reduce to
\beq\label{eucl1}
R^{\lambda\beta}_{\alpha\rho} R^{\rho\delta}_{\gamma\mu} = R^{\lambda\delta}_{\gamma\rho}R^{\rho \beta}_{\alpha\mu} \qquad \mbox{and} \qquad 
R^{\lambda\beta}_{\gamma\nu} R^{\mu\gamma}_{\alpha\rho} = 
R^{\mu\beta}_{\gamma\rho} R^{\lambda\gamma}_{\alpha\nu}
\eeq
that is the first relation of \eqref{yb1} and the first relation of \eqref{yb2}.
\end{prop}
\begin{proof}
We know from \eqref{real} that $(R^{-1})^{\beta\mu}_{\lambda\alpha}
=\overline{R}^{\mu\beta}_{\alpha\lambda}$. When comparing the first sum in the point (iii) 
of \eqref{cerad} with \eqref{real}, we see there is a `transposition'in the indices $\lambda, \mu$
and this leads to $R^{\lambda\beta}_{\alpha\mu}=(R^{-1})^{\beta\mu}_{\lambda\alpha}$. Similarly,
a comparison of the second sum in the point (iii) of \eqref{cerad} with \eqref{real}, shows 
a `transposition' in the indices $\alpha, \beta$ and this gives
$R^{\mu\alpha}_{\beta\lambda}=(R^{-1})^{\beta\mu}_{\lambda\alpha}$. 
The fact that then the relations \eqref{yb1} and \eqref{yb2} reduce to the two in \eqref{eucl1} is evident. 
\end{proof}
Finally the initial relations \eqref{alde} define (the algebra of) a noncommutative product of the Euclidean spaces $\mathbb R^{N_1}$ and $\mathbb R^{N_2}$  if and only if the matrix $R^{\lambda\alpha}_{\beta\mu}$ satisfy 
relations \eqref{eucl0} and \eqref{eucl1}.

 \subsection{Noncommutative spheres and products of spheres}
 With the quadratic elements $(x_1,x_1) = \sum^{N_1}_{\lambda=1}(x^\lambda_1)^2$ and 
$(x_2,x_2) = \sum^{N_2}_{\alpha=1}(x^\alpha_2)^2$ of $\car$ being central, one may consider the quotient algebra
$$
\car / (\{  (x_1,x_1)-\bbbone, (x_2,x_2)-\bbbone\})
$$
which defines by duality the noncommutative product 
$$
S^{N_1-1} \times_R S^{N_2-1}
$$
of the classical spheres $S^{N_1-1}$ and $S^{N_2-1}$. Indeed, for $R=R_0$, the above quotient is the restriction to $S^{N_1-1}\times S^{N_2-1}$ of the polynomial functions on 
$\IR^{N_1+N_2}$. 

Furthermore, with the central quadratic element $(x,x) = \sum^{N_1+N_2}_{a=1} (x^a)^2= (x_1,x_1)+(x_2,x_2)$,
one may also consider the quotient of $\car$
$$
\car/( (x,x)-\bbbone).
$$
This defines (by duality) the noncommutative $(N_1+N_2-1)$-sphere 
$$
S^{N_1+N_2-1}_R
$$
which is a noncommutative spherical manifold in the sense of \cite{ac-lan:2001}, \cite{ac-mdv:2002a} as explained below.

\section{Structure of the matrix $R$}

In this section we analyze the consequences of Proposition 3.4 for the structure of $R$.

\subsection{Consequences of the equations (3.16)} The consequences of the equations \eqref{eucl0} are summarized in the following proposition.
\begin{lemm}\label{ST}
The $R^{\lambda\alpha}_{\beta\mu}$ satisfy \eqref{eucl0} if and only if
\beq \label{eqSA}
R^{\lambda\alpha}_{\beta\mu}=S^{\lambda\alpha}_{\mu\beta}+iT^{\lambda\alpha}_{\mu\beta}
\eeq
with
\beq\label{eqS}
S^{\lambda\alpha}_{\mu\beta}=S^{\lambda\beta}_{\mu\alpha}=S^{\mu\alpha}_{\lambda\beta}=\bar S^{\lambda\alpha}_{\mu\beta}
\eeq
\beq\label{eqT}
T^{\lambda\alpha}_{\mu\beta}=-T^{\lambda\beta}_{\mu\alpha}=-T^{\mu\alpha}_{\lambda\beta}=\bar T^{\lambda\alpha}_{\mu\beta}
\eeq
and
\beq\label{eqino}
S^2+T^2=\bbbone_{N_1}\otimes \bbbone_{N_2}
\eeq
\beq\label{eqc}
[T,S]=0
\eeq
where $S$ and $T$ are the corresponding endomorphisms of $\mathbb R^{N_1}\otimes \mathbb R^{N_2}$.
\end{lemm}
\noindent \underbar{Proof.} The equations $R^{\lambda\alpha}_{\beta\mu}=R^{\mu\beta}_{\alpha\lambda}$ imply that one has  $R^{\lambda\alpha}_{\beta\mu}=S^{\lambda\alpha}_{\mu\beta}+ \tilde T^{\lambda\alpha}_{\mu\beta}$ with $S$ symmetric in $(\lambda,\mu)$ and in $(\alpha,\beta)$ and $\tilde T$ antisymmetric in $(\lambda,\mu)$ and in $(\alpha,\beta)$, while the equations $R^{\lambda\alpha}_{\beta\mu}=\bar R^{\mu\alpha}_{\beta\lambda}$ imply that the $S$ are real and the $\tilde T$ are imaginary, so $\tilde T=iT$ with $T$ real. Finally the equations $\bar R^{\mu\beta}_{\alpha\lambda}=(R^{-1})^{\beta\mu}_{\lambda\alpha}$ reads $(S+iT)(S-iT)=\bbbone_{N_1}\otimes \bbbone_{N_2}$ which implies $S^2+T^2=\bbbone_{N_1}\otimes \bbbone_{N_2}$ and $[T,S]=0$.  $\square$

\subsection{Consequences of the equations (3.17)} It follows from \eqref{eqS} that one has 
\beq\label{AB}
S=\sum_{r=1}^p A_r\otimes B_r
\eeq
where the $A_r$ are real $N_1\times N_1$ symmetric matrices and the $B_r$ are real $N_2\times N_2$ symmetric matrices. Furthermore, one can assume that the $A_r$ are linearly independent and that the $B_r$ are also linearly independent. Similarly, it follows from \eqref{eqT} that one has
\beq\label{CD}
T=\sum_{a=1}^q C_a\otimes D_a
\eeq
where the $C_a$ are real $N_1\times N_1$ antisymmetric matrices and the $D_a$ are real $N_2\times N_2$ antisymmetric matrices and one can assume that the $C_a$ are linearly independent as well as the $D_a$.

Thus by setting
\beq\label{chR}
\hat R^{\lambda\alpha}_{\mu\beta}= R^{\lambda\alpha}_{\beta\mu}
\eeq 
one has the representation
\beq\label{ABCD}
\hat R=\sum_r A_r\otimes B_r+i\sum_aC_a\otimes D_a
\eeq
for the endomorphism $\hat R$ of $\mathbb R^{N_1}\otimes \mathbb R^{N_2}$ 
with the $A_r$ linearly independent as well as the $B_r$, the $C_a$ and the $D_a$. Moreover, this representation \eqref{ABCD} for $R$ which is implied by \eqref{eqS} and \eqref{eqT} is in fact equivalent to \eqref{eqS} and \eqref{eqT} as easily seen.
\begin{lemm}\label{Qad}
The relations \eqref{eucl1} are equivalent to
\beq\label{AC}
[A_r,A_s]=0,\> [A_r,C_a]=0,\> [C_a,C_b]=0
\eeq
\beq\label{BD}
[B_r,B_s]=0,\> [B_r,D_a]=0,\> [D_a,D_b]=0
\eeq
for $r,s\in \{1,\dots, p\}$ and  $a,b\in\{1,\dots,q\}$.
\end{lemm}

\noindent \underbar{Proof}. The first relations \eqref{eucl1} read as endomorphisms of $\mathbb R^{N_1}\otimes \mathbb R^{N_2}\otimes \mathbb R^{N_2}$
\[
\begin{array}{ll}
\sum_{r,s}[A_r,A_s] \otimes B_r\otimes B_s & +\sum_{r,a}[A_r,C_a]\otimes (B_r\otimes D_a-D_a\otimes B_r)\\
\\
 & + \sum_{a,b}[C_a\otimes C_b] \otimes D_a\otimes D_b =0
\end{array}
\]
which is equivalent to \eqref{AC} in view of the assumptions of independence. Similarly the second relations \eqref{eucl1} are equivalent to \eqref{BD}. $\square$

It is clear that the representation \eqref{ABCD} with properties \eqref{AC} and \eqref{BD} imply the equation \eqref{eqc} and that one finally has the following theorem.

\begin{theo}\label{rR}
Relations \eqref{alde} define (the algebra of) a noncommutative product of the Euclidean space $\mathbb R^{N_1}$ with the Euclidean space $\mathbb R^{N_2}$ if and only if $\hat R$ defined by \eqref{chR} admits the representation \eqref{ABCD} above with the relations \eqref{AC}, \eqref{BD} and the normalization condition \eqref{eqino} for $S$ and $T$ given by \eqref{AB} and \eqref{CD}.
\end{theo}

Note that the normalization condition \eqref{eqino} reads
\beq\label{NABCD}
\sum_{r,s} A_rA_s\otimes B_rB_s + \sum_{a,b} C_a C_b\otimes D_aD_b=\bbbone_{N_1} \otimes \bbbone_{N_2}
\eeq
in terms of the representation \eqref{ABCD}.

\subsection{Reduction to canonical form} It is clear that one does not change the $\ast$-algebra $\ca_R$ and its central elements $\sum_\lambda(x^\lambda_1)^2$ and $\sum_\alpha (x^\alpha_2)^2$ if one changes the generators $x^\lambda_1$, $x^\alpha_2$ by a rotation of $O(N_1)\times O(N_2)$. We now use this freedom to put the $A_r,C_a$ and the $B_r,D_a$ in canonical form by using their commutation properties \eqref{AC} and \eqref{BD}.

Note that $O(N_1)$ with $A_r,C_a$ satisfying \eqref{AC} is completely similar to $O(N_2)$ with $B_r,D_a$ satisfying \eqref{BD}, so it is sufficient to treat the case of the $A_r, C_a$ satisfying \eqref{AC} with the action of $O(N_1)$.

The $A_r,C_a$ are a family of real $N_1\times N_1$-matrices which commute among themselves with the $A_r$ symmetric while the $C_a$ are antisymmetric. It follows that the $A_r$ and $(C_a)^2$ is a family of commuting real symmetric $N_1\times N_1$-matrices which can be simultaneously diagonalized by a rotation of $O(N_1)$. Notice that the $(C_a)^2$ are negative since $(C_a)^2=-C_aC^t_a$.

Let $-(c_a)^2<0$ be a non trivial eigenvalue of $(C_a)^2$ with eigenvector $x$,  then $C_ax$ is not vanishing and is again an eigenvector of $(C_a)^2$ with eigenvalue $-(c_a)^2$ which is orthogonal to $x$ since 
\[
(x, C_ax)=(C^t_a x, x)=-(C_a x,x)=0
\]
 in view of the antisymmetry of $C_a$.
 
  To the 2-dimensional cell $-(c_a)^2\left(\begin{array}{cc}1 & 0\\ 0 & 1\end{array}\right)=-(c_a)^2 I$ occurring in the diagonalization of $(C_a)^2$ corresponds the cell $c_a\left(\begin{array}{cc} 0 & -1\\1 & 0\end{array}\right)=c_aJ$ of $C_a$ in the plane spanned by $x$ and $C_ax$ (where one can choose $c_a$ to be positive, by an eventual transposition).

Finally, by using the fact that the $A_r$ and the $C_a$ commute, there is an even dimensional subspace of $\mathbb R^{N_1}$, say of dimension $2k_1\leq N_1$, such that the diagonalization of the $A_r$ and the $(C_a)^2$ leads to an orthonormal frame in which one has
\beq\label{aA}
A_r = \left(\begin{array}{llllll}
a^1_r I & & & & &\hfill\cr
 & \ddots & & & &\hfill\cr
 & & a^{k_1}_r I & & &\hfill\cr
 & & & a^{' 1}_r & &\hfill\cr
 &&&& \ddots &\hfill\cr
 &&&&& a^{' N_1-2k_1}_r\hfill\cr
\end{array}
 \right)
\eeq
 and
\beq\label{cC}
C_a = \left(\begin{array}{llllll}
c^1_aJ& & & & &\hfill\cr
 & \ddots & & & &\hfill\cr
 & & c^{k_1}_aJ & & &\hfill\cr
 & & & 0& &\hfill\cr
 &&&& \ddots &\hfill\cr
 &&&&& 0\hfill\cr
\end{array}
 \right)
\eeq
with 
\beq\label{IJ}
I=\left(\begin{array}{cc}
1 & 0\\
0 & 1\end{array}\right),\>\>\>\>
J=\left(\begin{array}{cc}
0 & -1\\
1 & 0\end{array}\right)
\eeq
and $a^m_r, a^{'m'}_r\in \mathbb R,\>\> c^m_a\in \mathbb R$ and $\sum_a(c^m_a)^2\not=0$ for $m\in \{1,\dots,k_1\}, m'\in \{2k_1+1,\dots,N_1\}$.

The same discussion leads to the existence of an orthonormal frame in $\mathbb R^{N_2}$ in which one has

\beq\label{bB}
B_r = \left(\begin{array}{llllll}
b^1_r I & & & & &\hfill\cr
 & \ddots & & & &\hfill\cr
 & & b^{k_2}_rI & & &\hfill\cr
 & & & b^{' 1}_r & &\hfill\cr
 &&&& \ddots &\hfill\cr
 &&&&& b^{' N_2-2k_2}_r\hfill\cr
\end{array}
\right)
\eeq
and
\beq\label{dD}
D_a = \left(\begin{array}{llllll}
d^1_aJ& & & & &\hfill\cr
 & \ddots & & & &\hfill\cr
 & & d^{k_2}_aJ & & &\hfill\cr
 & & & 0& &\hfill\cr
 &&&& \ddots &\hfill\cr
 &&&&& 0\hfill\cr
\end{array}
\right)
\eeq
with $b^m_r,b^{'m'}_r\in \mathbb R, d^m_a\in \mathbb R$ and $\sum_a(d^m_a)^2\not=0$ for $m\in\{1,\dots, k_2\}$, $m'\in\{2k_2+1,\dots, N_2\}$.

In both cases by eventual transpositions one can choose that the $c^i_a$ and the $d^j_a$ are positive.

\subsection{Identification as ``$\theta$-deformation"}

The normalization condition which reads \eqref{NABCD} in terms of the $A_r, B_r, C_a$ and $D_a$ are equivalent to the following conditions for the $A_r, B_r,  C_a, D_a$ given by \eqref{aA}, \eqref{bB}, \eqref{cC}, \eqref{dD} :
\beq\label{abcd}
\left(\sum_r a^{m_1}_r b^{m_2}_r\right)^2 +\left(\sum_a c^{m_1}_a d^{m_2}_a\right)^2=1
\eeq

\beq \label{ab1}
\left(\sum_r a^{m_1}_r b'^{m'_2}_r\right)^2=1
\eeq

\beq\label{ab2}
\left(\sum_r a'^{m'_1}_r b^{m_2}_r\right)^2=1
\eeq
\beq\label{ab3}
\left(\sum_r a'^{m'_1}_r b'^{m'_2}_r\right)^2=1
\eeq

for $1\leq m_1\leq k_1,\>\> 1\leq m_2\leq k_2,\>\>  2k_1+1\leq m'_1\leq N_1,\>\> 2k_2+1\leq m'_2\leq N_2.$

It follows from \eqref{abcd} that one can set
\beq\label{theta}
\sum_r a^{m_1}_r b^{m_2}_r=\cos(\theta_{m_1m_2}), \sum_a c^{m_1}_a d^{m_2}_a=\sin(\theta_{m_1m_2})
\eeq
and  the other equations imply that one has
\beq\label{epsi}
\sum_r a^{m_1}_r b'^{m'_2}_r = \varepsilon^{m_1m'_1}, \sum_r a'^{m'_1}_r b^{m_2}_r=\varepsilon^{m'_1m_2},\>\> \sum_r a'^{m'_1}_r b'^{m'_2}_r=\varepsilon^{m'_1m'_2}
\eeq
where the $\varepsilon^{uv}$ are $\pm 1$.

 Let $x^\lambda_1$ and $x^\alpha_2$ be an orthonormal basis for the above canonical form and let us introduce complex coordinates $z^m_1=x^{2m-1}_1 +ix^{2m}_1$ for $m\in \{1\dots, k_1\}$ and $z^m_2=x^{2m-1}_2+ix^{2m}_2$ for $m\in \{1,\dots, k_2\}$. Then the relations for the $x^\lambda_1$ and $x^\alpha_2$
  read
\beq\label{Ck1Ck2}
\left\{
\begin{array}{l}
z^{m_1}_1 z^{m_2}_2 = e^{i\theta_{m_1m_2}} z^{m_2}_2 z^{m_1}_1\\
\\
\bar z^{m_1}_1 z^{m_2}_2=e^{-i\theta_{m_1m_2}}z^{m_2}_2 \bar z^{m_1}_1\\
\\
z^{m_1}_1 \bar z^{m_2}_2= e^{-i\theta_{m_1m_2}}\bar z_2^{m_2} z_1^{m_1}\\
\\
\bar z_1^{m_1} \bar z^{m_2}_2 = e^{i\theta_{m_1 m_2}} \bar z_2^{m_2} \bar z_1^{m_1}
\end{array}
\right.
\eeq	
for $m_1\in \{1,\dots,k_1\},\>\>  m_2\in\{1,\dots,k_2\}$ while the other relations read
\beq\label{pcl1}
z^{m_1}_1 x^{m'_2}_2=\varepsilon^{m_1m'_2} x^{m'_2}_2 z_1^{m_1}
\eeq
for $m_1\in \{1,\dots, k_1\},\>\> m'_2\in \{2k_2+1,\dots, N_2\}$
\beq\label{pcl2}
x^{m'_1}_1 z^{m_2}_2=\varepsilon^{m'_1m_2} z^{m_2}_2 x^{m'_1}_1
\eeq
for $m'_1\in \{2k_1+1,\dots, N_1\},\>\> m_2\in \{1,\dots, k_2\}$ and 
\beq\label{pcl3}
x^{m'_1}_1 x^{m'_2}_2=\varepsilon^{m'_1 m'_2} x^{m'_2}_2 x^{m'_1}_1
\eeq
for $m'_1\in \{2k_1+1,\dots, N_1\},\>\>  m'_2\in \{2k_2+1,\dots, N_2\}$.

Relations \eqref{Ck1Ck2} define (by duality) a noncommutative product of the Hilbertian spaces $\mathbb C^{k_1}$ and $\mathbb C^{k_2}$ which is a particular case of noncommutative versions of $\mathbb C^{k_1+k_2}$ known as $\theta$-deformations (see \S 5.1 below for the case $k_1+k_2=2$)  and studied in details in \cite{ac-mdv:2002a} where the corresponding generalized Clifford algebras are given in explicit form and used there to show that the corresponding noncommutative spheres $S^{2(k_1+k_2)-1}$ are noncommutative spherical manifolds in the sense of \cite{ac-lan:2001}, \cite{ac-mdv:2002a}. This in turn easily implies that the noncommutative sphere $S^{N_1+N_2-1}$ of \S 3.6 are noncommutative spherical manifolds as well; this is clear if all the $\varepsilon^{uv}$ occurring in formulae \eqref{epsi} are $+1$ and it is not hard to take care of the eventual $ -1$ cases. Alternatively we could stay in the connected component which contains the classical (commutative) solution $R_0$ given by \eqref{classic}.

\section{Noncommutative quaternionic planes and tori}

The aim of this section is to define noncommutative quaternionic planes by using the above description of noncommutative Euclidean spaces. As a preliminary we first describe the noncommutative complex plane $\mathbb C^2_\theta$ in terms of noncommutative product of $\mathbb R^2$ with $\mathbb R^2$.

\subsection{The noncommutative complex planes $\mathbb C^2_\theta$}
A one-parameter family of noncommutative two-dimensional complex spaces $\Ct$ was introduced in \cite{ac-lan:2001}. With $\theta\in \IR$, the coordinate unital $*$-algebra  $\ca_\theta$ of $\Ct$
is generated by two normal elements $z_1, z_2$, that is 
\beq\label{com1}
[z_1, z^*_1]=0=[z_2, z^*_2]
\eeq
with relations:
\beq\label{com2}
z_1 z_2 = e^{i \theta} z_2 z_1 , \qquad z_1 z^*_2 = e^{- i \theta} z^*_2 z_1 
\eeq
together with their $*$-conjugates. Sending $\theta \to- \theta$ results into an isomorphic algebra at the expenses of exchanging 
$z_1 \leftrightarrow z_2$. Thus this family is really parametrised by 
$\IS^1 / \IZ_2$ that is by the real projective line $P_1(\mathbb R)$. 

One easily checks that both $z_1 z^*_1$ and $z_2 z^*_2$ are in the center of the algebra $\ca_\theta$, so one may define the algebra $\ca_\theta/(\{ z_1z_1^\ast-\bbbone, z_2z_2^\ast-\bbbone\})$ which is the algebra defining the noncommutative torus $T^2_\theta$.

One can pass to a noncommutative four-dimensional Euclidean spaces 
$\Ct \simeq \Rt$ via hermitian generators $(x_1^1, x_1^2)$ and $(x_2^1, x_2^2)$
by writing
$$
z_1 = x_1^1 + i \, x_1^2 , \qquad z_2 = x_2^1 + i \, x_2^2 .
$$
Then the algebra relations are easily found to be given by 
\[
 x_1^1 x_1^2 = x_1^2 x_1^1 , \qquad x_2^1 x_2^2 = x_2^2 x_2^1  \nn , \qquad
 x_1^\lambda x_2^\alpha = R^{\lambda \alpha}_{\beta \mu} \, x_2^\beta x_1^\mu 
\]
with matrix 
\beq\label{r4d}
R^{\lambda \alpha}_{\beta \rho} = \cos(\theta) \, \delta^\lambda_\rho \delta^\alpha_\beta + 
i \sin(\theta) \, J^\lambda_\rho J^\alpha_\beta \qquad \mbox{with} \quad 
J = \begin{pmatrix} 
0 & - 1 \\
1 & 0 
\end{pmatrix}
\eeq
so one has $\ca_\theta=\ca_R$ with the above $R$ matrix and in fact, in the notation of \S 3.3 we are really defining $\mathbb R^4_\theta$ as a noncommutative product 
$\Rt=\IR^2 \times_\theta \IR^2$  of the Euclidean space $\mathbb R^2$ with itself. Indeed Formula \eqref{r4d} is a particular case of Formula \eqref{ABCD} and all the conditions of Theorem \ref{rR} are satisfied with $N_1=N_2=2$.

It is worth noticing here that the matrix $J$ in \eqref{r4d} represents the complex imaginary unit $i$. Indeed one has $i(x^1+ix^2)=(-x^2+ix^1)$ so the multiplication by $i$ reads
\[
\left(\begin{array}{c}
x^1\\ x^2\end{array}\right)\mapsto J \left(\begin{array}{c}x^1\\ x^2\end{array}\right)=\left(\begin{array}{c}-x^2\\ x^1\end{array}\right)
\]
in terms of the coordinates of $\mathbb R^2$.\\

In the following we shall generalize Formula \eqref{r4d} for the quaternions so in the next subsection we give a description of the imaginary quaternionic unit in terms of real $4\times 4$-matrices acting on the coordinates of $\mathbb R^4$ corresponding to quaternions.\\

\noindent\underbar{Remark}. Notice that $(z_1,z_2)\mapsto (z'_1,z'_2)$ with $z'_1=y_1z_1,z'_2=y_2z_2$ for $(y_1,y_2)\in \mathbb C^2$ preserve the relations \eqref{com1} and \eqref{com2} which in particular implies that one has an action of the multiplicative group $\mathbb C_\ast\times \mathbb C_\ast$ by automorphisms on the $\ast$-algebra $\ca_R$ of coordinates of $\mathbb C^2_\theta$, where $\mathbb C_\ast$ is the multiplicative group of non vanishing element of $\mathbb C$ (i.e.$\mathbb C_\ast=\mathbb C\backslash \{0\})$ and where $R$ is given by \eqref{r4d}. This induces an action of 
\[
T^2=U_1(\mathbb C)\times U_1(\mathbb C)= S^1\times S^1
\]
 on $\ca_R$ which induces an action of $T^2$ on the corresponding noncommutative torus $T^2_\theta$.

\subsection{Quaternions and the 4-dimensional Euclidean space}

The space of quaternions $\IH$ is identified with $\IR^4$ in the usual way:
\beq\label{qrid}
\IH \ni q = x^0 e_0 + x^1 e_1 + x^2 e_2 + x^3 e_3 \quad \longmapsto \quad x = (x^\mu) = 
(x^0, x^1, x^2, x^3) \in \IR^4 
\eeq
where $e_0=1$ and the imaginary units $e_a$ obey the multiplication rule 
$$
e_a e_b = - \delta_{ab} + \sum_{c=1}^3\varepsilon_{abc} e_c 
$$
with standard notations. 

The subgroup $U_1(\mathbb H)=\{q\vert q\bar q=1\}$ of the multiplicative group $\mathbb H_\ast$ of non vanishing quaternions is isomorphic to $SU(2)$ and will be identified to it
\beq\label{su2}
U_1(\mathbb H)=SU(2)
\eeq
It can also  be identified to the sphere $S^3=\{x \mathbb \in R^4\vert \parallel x\parallel^2=1\}$.

With the identification \eqref{qrid}, left and right multiplication of quaternions are represented 
by matrices acting on $\IR^4$
$$
L_{q'} q := q'q \quad \to \quad E^+_{q'} (x) \qquad \mbox{and} \qquad 
R_{q'} q := qq' \quad \to \quad E^-_{q'} (x) 
$$
with $q,q'\in \mathbb H$ and $x\in \mathbb R^4$ corresponding to $q$ as before.
For $q\in U_1(\mathbb H)$, both $E^+_{q}$ and $E^-_{q}$ are orthogonal matrices. Thus by restriction to $q\in U_1(\mathbb H)$ the mappings $q\mapsto E^{\pm}_q$ give two commuting representations of $U_1(\mathbb H)=SU(2)$ in $\mathbb R^4$ by rotations ($\in SO(4)$). Taken together this gives an action of $\SU(2)_L \times \SU(2)_R$ on $\IR^4$, with $L/R$ denoting left/right action.
This action is  the action of $\SO(4) = \SU(2)_L \times \SU(2)_R / \IZ_2$. 

Let us denote $E^{\pm}_{a} = E^{\pm}_{e_a}$ for the imaginary units. 
By definition one has that 
\begin{align*}
E^{+}_{a} E^{-}_{b} = E^{-}_{b} E^{+}_{a} , \qquad  E^{\pm}_{a} E^{\pm}_{b} = - \delta_{ab}\bbbone 
\pm \sum_{c=1}^3\varepsilon_{abc} E^{\pm}_{c} 
\end{align*}
but in the following it will turn out to be more convenient to change a sign to the `right' matrices: we shall rather use matrices 
$J^{+}_{a}:=E^{+}_{a}$ and $J^{-}_{a}:=-E^{-}_{a}$. For these one has
\begin{align*}
J^{+}_{a} J^{-}_{b} = J^{-}_{b} J^{+}_{a} , \qquad 
J^{\pm}_{a} J^{\pm}_{b} = - \delta_{ab} \bbbone + \sum_{c=1}^3\varepsilon_{abc} J^{\pm}_{c} 
\end{align*}
that is the matrices $J^\pm_a$ are two commuting copies of the quaternionic imaginary units.

By using the above definitions one can explicitly compute the expressions
\beq
(J^{\pm}_{a})_{\mu\nu} = \mp ( \delta_{0\mu}  \delta_{a\nu} - \delta_{a\mu}  \delta_{0\nu})
+ \sum_{b,c=1}^3\varepsilon_{abc} \delta_{b\mu} \delta_{c\nu}  \quad \mbox{for} \quad a=1,2,3 
\eeq
for the components of the real $4 \times 4$  matrices $J^{\pm}_{a}$.

For the standard positive definite scalar product on $\IR^4$, the six matrices $J^{\pm}_{a}$ are readily checked to be antisymmetric, ${}^t J^{\pm}_{a} = - J^{\pm}_{a}$ and one finds in addition that the $J^+_a$ form an orthonormal basis of the space of antisymmetric self-dual matrices, the $J^-_a$ form an orthonormal basis of the antisymmetric anti-self-dual matrices while 
the nine matrices $J^{+}_{a} J^{-}_{b}$ form an orthonormal basis for the space of symmetric traceless matrices. In fact, the data $(\bbbone, J^{+}_{a}, J^{-}_{a}, J^{+}_{a} J^{-}_{b})$ form an orthonormal basis of the endomorphisms algebra of $\IR^4$ (that is matrices) adapted to the decomposition 
\begin{align*}
\IR^4 \otimes \IR^4{}^* \simeq \IR^4{}^* \otimes \IR^4 & = D^{(0,0)} \oplus D^{(1,0)} \oplus D^{(0,1)} \oplus D^{(1,1)} \\
& = \IR \bbbone \oplus \wedge^2_+ \IR^4{}^* \oplus   \wedge^2_- \IR^4{}^* \oplus S^2_{(0)} \IR^4{}^* 
\end{align*}
into irreducible $\SO(4)$-invariant subspaces. Here $S^2_{(0)} \IR^4{}^*$ is the space of trace-less elements of the degree-two part of the symmetric algebra over $\IR^4{}^*$, while the space $\wedge^2 \IR^4{}^*$ of exterior two-forms (anti-symmetric two-vectors) on $\IR^4$ is split into a 
self-dual and an anti-self-dual part,\linebreak[4]
$\wedge^2 \IR^4{}^* = \wedge^2_+ \IR^4{}^* \oplus   \wedge^2_- \IR^4{}^*$, where
$
\wedge^2_\pm \IR^4{}^* = \{F \in \wedge^2 \IR^4{}^* \, | \, \star F = \pm F \} \, \qquad $ with
$(\star F)_{\mu \nu} = \tfrac{1}{2} \sum_{\rho, \sigma=1}^4 \varepsilon_{\mu \nu \rho \sigma} F_{\rho \sigma} .
$ 
For the space $\wedge^2 \IR^4{}^*$ with the scalar product 
$(F | F') = \tfrac{1}{4} \sum_{\mu, \nu=1}^4 F_{\mu \nu} F'_{\mu \nu}$, the splitting 
$\wedge^2 \IR^4{}^* = \wedge^2_+ \IR^4{}^* \oplus \wedge^2_- \IR^4{}^*$ is an orthogonal one into three-dimensional irreducible subspaces $\wedge^2_\pm \IR^4{}^*$ and the representation of $\SO(4)$ on 
$\wedge^2 \IR^4{}^* = \wedge^2_+ \IR^4{}^* \oplus \wedge^2_- \IR^4{}^*$ induces an homomorphism, 
\beq\label{act4}
\pi: \SO(4) \to \SO(3) \times \SO(3) 
\eeq
with kernel $\pm\bbbone$.

\noindent\underbar{Remark}.  The two $SO(3)$ group in \eqref{act4} are the isomorphic images of the two $SO(3)$ Lie subgroups $SO_+(3)$ and $SO_-(3)$ of $SO(4)$ with Lie algebras given by
\beq\label{LSO}
\Lie(SO_{\pm}(3))=\left\{\sum_a y^aJ^{\pm}_a\vert (y^1, y^2 ,y^3)\in \mathbb R^3\right\}
\eeq
which commute in $\Lie(SO(4))$.

It is clear from above that under the adjoint action of $SO(4)$, $SO_-(3)$ preserves the $J^+_a$ while $SO_+(3)$ preserves the $J^-_a$.

\subsection{Noncommutative quaternionic planes and tori}
We are now ready to give a definition of what is meant by a ``noncommutative product" $\mathbb H\times_R \mathbb H$ of $\mathbb H$ with $\mathbb H$. We impose of course that its coordinate algebra is a $\ast$-algebra $\ca_R$ corresponding to a noncommutative product of two 4-dimensional Euclidean spaces $\mathbb R^4$. But now we want more namely to be able to act say on the left by pairs of quaternions $q_1$ and $q_2$ on the quaternions corresponding to $x_1\in \mathbb R^4$ and $x_2\in \mathbb R^4$ as in \eqref{qrid}. From last section we know that this corresponds to $x_1\mapsto (q^0_1\bbbone +q^a_1 J^+_a)x_1$ and \linebreak[4] $x_2\mapsto (q^0_2\bbbone +q^a_2J^+_a) x_2$. 

Thus we now impose that $x_1\mapsto J^+_ax_1$  and $x_2\mapsto J^+_b x_2$
induces an automorphism of $\ca_R$ for any  $a,b\in\{1,2,3\}$. The most natural way to generalize \eqref{r4d} in this setting is to define $\ca_R$ as the $\ast$-algebra generated by the hermitian elements $x^\lambda_1$ and $x^\alpha_2$ ($\lambda,\alpha\in\{0,1,2,3\}$) with relations
\beq\label{ald4}
x^\lambda_1 x^\mu_1=x^\mu_1 x^\lambda_1,\qquad  x^\alpha_2x^\beta_2=x^\beta_2x^\alpha_2,\qquad  x^\lambda_1 x^\alpha_2=R^{\lambda\alpha}_{\beta\mu} x^\beta_2 x^\mu_1
\eeq
with 
\beq\label{r8d}
R^{\lambda\alpha}_{\beta\mu} = u\delta^\lambda_\mu \delta^\alpha_\beta +iJ^-(\vec v_1)^\lambda_\mu J^-(\vec v_2)^\alpha_\beta
\eeq
and 
\beq\label{norm8}
(u)^2+\parallel \vec v_1\parallel ^2\parallel \vec v_2\parallel^2=1
\eeq
where $J^-(\vec v)= v^a J^-_a$ and $\parallel \vec v\parallel^2=\sum_a(v^a)^2$ for $\vec v = (v^1,v^2,v^3)\in \mathbb R^3$.
 
 Indeed then the mappings $x_1\mapsto J^+_a x_1$, $x_2\mapsto J^+_bx_2$ for $a,b\in \{1,2,3\}$ leave the relations \eqref{ald4} of $\ca_R$ invariant and define therefore $\ast$-automor\-phisms of $\ca_R$ in view of the commutations of the $J^+_a$ with the $J^-_b$ for $a,b\in \{1,2,3\}$.
 
 It follows that the multiplicative group $\mathbb H_\ast\times \mathbb H_\ast$ acts by automorphisms of the $\ast$-algebra $\ca_R$, setting $q=q^0+q^a e_a\in \mathbb H_\ast\mapsto q^0\bbbone +q^aJ^+_a$ with obvious conventions. This induces an action of $U_1(\mathbb H)\times U_1(\mathbb H)$ on $\ca_R$ by restriction to the $q\in U_1(\mathbb H)$.  This action passes to the quotient by the ideal generated by the two central elements $\sum_\lambda (x^\lambda_1)^2$, $\sum_\alpha (x^\alpha_2)^2$ and defines an action of the classical quaternionic torus  $T^2_\mathbb H=U_1(\mathbb H)\times U_1(\mathbb H)$ by $\ast$-automorphisms of the coordinate algebra of the ``noncommutative" quaternionic torus $(T^2_{\mathbb H})_R$. The situation is completely similar to the one of the complex case leading to an action of the classical torus $T^2$ on the noncommutative torus $T^2_\theta$.
 
 One can simplify the matrix $R$ of \eqref{r8d} by isomorphisms which preserve the actions of $\mathbb H_\ast\times \mathbb H_\ast$ on $\ca_R$ i.e. by the action of $SO_-(3)$ described in the last section since it preserves the $J^+_a$  (i.e. the action of $\mathbb H_\ast\times \mathbb H_\ast$). With this,  one can modify by rotation $\vec v_1$ and $\vec v_2$ so that one can arrive to the replacement $J^-(\vec v_1)\otimes J^-(\vec v_2)$ by say $J^-_1\otimes (u^1J^-_1+u^2J^-_2)$ with $(u^1)^2+(u^2)^2=\parallel \vec v_1\parallel^2 \parallel \vec v_2\parallel^2$. By setting $u=u^0$ in \eqref{r8d} this leads to 
 \beq\label{sr8d}
 R^{\lambda\alpha}_{\beta\mu}=u^0\delta^\lambda_\mu\delta^\alpha_\beta+iJ^{-\lambda}_{1\mu} (u^1J^{-\alpha}_{1\beta}+u^2J_{2\beta}^{-\alpha})
 \eeq
for the simplified $R$ matrix. Condition \eqref{norm8} becomes
\[
(u^0)^2+(u^1)^2+(u^2)^2=1
\]
which gives a parametrization by $S^2=P_1(\mathbb C)$ the complex projective line.

The above construction of ``noncommutative products" $\mathbb H\times_R \mathbb H$ and noncommutative quaternionic tori $(T^2_{\mathbb H})_R$ with action of the classical quaternionic torus $T^2_{\mathbb H}=U_1(\mathbb H)\times U_1(\mathbb H)= S^3\times S^3=SU(2)\times SU(2)$ seems to be quite optimal.

\section{Conclusion}

In Section 2 (\S 2.3) we have defined a noncommutative $n$-dimensional Euclidean space $\mathbb R^n$ by setting that its (noncommutative) coordinate algebra (the analog of the algebra of complex polynomials on $\mathbb R^n$) is a complex quadratic $\ast$-algebra $\ca$ generated be $n$ hermitian elements $x^k$ ($k\in\{1,\dots,n\}$) which is a Koszul Artin-Shelter algebra of global dimension $n$ and such that $\sum_k(x^k)^2$ is an element of the center $Z(\ca)$ of $\ca$. Let us discuss these assumptions for $\ca$.

Concerning the last point, centrality of $\sum_k(x^k)^2$, it is clear that we have in mind that the $x^k$ are ``orthonormal coordinates" with ``metric" $\sum_k x^k\otimes x^k$.

There is then a very natural definition of the corresponding generalization $C\ell(\ca)$ of the Clifford algebra as nonhomogeneous quadratic $\ast$-algebra with the Koszul dual $\ca^!$ of $\ca$ as homogeneous part and we have imposed the Koszul property for $\ca$ (or equivalently for $\ca^!$) to be sure that Condition \ref{i-ii} of Appendix A which is trivialy satisfied here implies the Poincar\'e-Birkhoff-Witt property for $C\ell(\ca)$. At this point it is worth noticing that this condition of Koszulity can be slightly weakened as pointed out in \cite{cas-she:2007} by imposing only
\beq\label{CS2}
\Ext^{3,3}_\ca(\mathbb C,\mathbb C)\not= 0\ \text{and}\ \Ext^{3,d}_\ca(\mathbb C,\mathbb C)=0\ \text{for}\ d>3
\eeq
but, on the other hand the Koszulity, which is equivalent to the purity assumption for $\Ext_\ca(\mathbb C,\mathbb C)$, is a very natural and desirable assumption here.

Finally the Gorenstein property is a well-suited noncommutative generalization of the Poincar\'e duality and the polynomial growth condition is adapted to the connexion between the number of independent generators and the global dimension that we have imposed (which is natural in the quadratic framework).

Let us notice that our definition of the noncommutative products of the Euclidean spaces $\mathbb R^{N_1}$ and $\mathbb R^{N_2}$ is tailored for the above definition of noncommutative Euclidean spaces $\mathbb R^n$: Our noncommutative products of $\mathbb R^{N—1}$ and $\mathbb R^{N_2}$ are noncommutative Euclidean spaces $\mathbb R^{N_1+N_2}$.

We could expect slightly more for the properties of the algebras $\ca$ corresponding to noncommutative Euclidean spaces $\mathbb R^n$. For instance it is natural to expect that the noncommutative spheres $S^{n-1}$ corresponding to $\ca/(\sum_k(x^k)^2-\bbbone)$ are noncommutative spherical manifolds in the sense of \cite{ac-lan:2001}, \cite{ac-mdv:2002a}. This would mean in view of Corollary \ref{Rprime}, to prove or assume that appropriate traces of products of the $\Gamma_k$ vanish to get the $K$-homological conditions of \cite{ac-lan:2001}, \cite{ac-mdv:2002a}. On the other hand, the homogeneization of these $K$-homological conditions leads to noncommutative Euclidean $\mathbb R^n$ in the above sense for all the cases which have been analyzed, in particular for all the cases with $n\leq 4$ and for all $n\in \mathbb N$ in the cases of the so-called $\theta$-deformations \cite{ac-mdv:2002a}. One may expect that it is always true and even that it leads to Calabi-Yau algebras \cite{gin:2006}; this latter point would mean in our quadratic Koszul Artin-Schelter context that the Nakayama automorphisms of the Koszul Frobenius algebras $\ca^!$ are trivial (i.e. are the identity mappings of the $\ca^!$).

Thus there remain a lot of open questions around the definition and the properties of the noncommutative Euclidean spaces.

\appendix

\section{Quadratic algebras}\label{quad}

To be definite we take the ground field to be complex numbers $\IC$. 
A homogeneous \emph{quadratic algebra} \cite{man:1988}, \cite{pol-pos:2005}.
is an associative algebra $\ca$ of the form
$$
\ca=A(E,R)=T(E)/(R)
$$
with $E$ a finite-dimensional vector space, $R$ a subspace of $E\otimes E$ and $(R)$ denoting the two-sided ideal of the tensor algebra $T(E)$ over $E$ generated by $R$. The space $E$ is the space of generators of $\ca$ and the subspace $R$ of $E\otimes E$ is the space of relations of $\ca$. The algebra $\ca=A(E,R)$ is naturally a graded algebra $\ca=\bigoplus_{n\in \mathbb N}\ca_n$ which is connected, that is such that $\ca_0=\IC i$ and generated by the degree 1 part, $\ca_1=E$.

To a quadratic algebra $\ca=A(E,R)$ as above one associates another quadratic algebra, \emph{its Koszul dual} $\ca^!$, defined by
$$
\ca^!=A(E^\ast,R^\perp)
$$
where $E^\ast$ denotes the dual vector space of $E$ and $R^\perp\subset E^\ast \otimes E^\ast$ is the orthogonal of the space of relations $R\subset E\otimes E$ defined by 
$$
R^\perp = \{\omega \in E^\ast\otimes E^\ast \, ; \, \langle\omega,r\rangle=0,\forall r\in R\} .
$$
As usual, by using the finite-dimensionality of $E$, one identifies $E^\ast \otimes E^\ast$ with the dual vector space $(E\otimes E)^\ast$ of $E\otimes E$. One has of course $(\ca^!)^!=\ca$
and the dual vector spaces $\ca^{!\ast}_n$ of the homogeneous components $\ca^!_n$ of $\ca^!$ are 
\beq
\ca^{!\ast}_1=E \qquad 
\mbox{and} \qquad 
\ca^{!\ast}_n=\bigcap_{r+s+2=n} E^{\otimes^r}\otimes R \otimes E^{\otimes^s}
\label{En}
\eeq
for $n\geq 2$, as easily verified. In particular $\ca^{!\ast}_2=R$ and $\ca^{!\ast}_n\subset E^{\otimes^n}$ for any $n\in \mathbb N$.

Consider the sequence of free left $\ca$-modules
\beq
K(\ca):\,  \cdots \stackrel{b}{\rightarrow}\ca\otimes \ca^{!\ast}_{n+1}\stackrel{b}{\rightarrow}\ca\otimes \ca^{!\ast}_n\rightarrow \dots \rightarrow \ca\otimes \ca^{!\ast}_2\stackrel{b}{\rightarrow}\ca\otimes E\stackrel{b}{\rightarrow}\ca\rightarrow 0
\label{K}
\eeq
where $b:\ca\otimes \ca^{!\ast}_{n+1}\rightarrow \ca\otimes \ca^{!\ast}_n$ is induced by the left $\ca$-module homomorphism of $\ca\otimes E^{\otimes^{n+1}}$ into $\ca\otimes E^{\otimes^n}$ defined by
$$
b(a\otimes (x_0\otimes x_1\otimes \dots \otimes x_n))=ax_0 \otimes (x_1\otimes \dots \otimes x_n)
$$
for $a\in \ca$, $x_k \in E$. It follows from \eqref{En} that $\ca^{!\ast}_n\subset R\otimes E^{\otimes^{n-2}}$ for $n\geq 2$, which implies that $b^2=0$. As a consequence, the sequence $K(\ca)$ in \eqref{K} is a chain complex of free left $\ca$-modules called the \emph{Koszul complex} of the quadratic algebra $\ca$. 
The quadratic algebra $\ca$ is said to be a \emph{Koszul algebra} whenever its Koszul complex is acyclic in positive degrees, that is, whenever $H_n(K(\ca))=0$ for $n\geq 1$. One shows easily that $\ca$ is a Koszul algebra if and only if its Koszul dual $\ca^!$ is a Koszul algebra.

It is important to realize that the presentation of $\ca$ by generators and relations is equivalent to the exactness of the sequence
$$
\ca\otimes R \stackrel{b}{\rightarrow} \ca\otimes E \stackrel{b}{\rightarrow}\ca \stackrel{\varepsilon}{\rightarrow} \IC\rightarrow 0
$$
with $\varepsilon$ the map induced by the projection onto degree $0$. Thus one always has
$$
H_1(K(\ca))=0 \qquad \mbox{and} \qquad H_0(K(\ca))=\IC
$$
and, whenever $\ca$ is Koszul, the sequence
$$
K(\ca)\stackrel{\varepsilon}{\rightarrow} \IC \rightarrow 0 
$$
is a free resolution of the trivial module $\IC$. This resolution is then a minimal projective resolution of $\IC$ in the category of graded modules  \cite{car:1958}.

Let $\ca=A(E,R)$ be a quadratic Koszul algebra such that $\ca^!_D\not=0$ and $\ca^!_n=0$ for $n>D$. 
Then the trivial (left) module $\IC$ has projective dimension $D$ which implies that $\ca$ has global dimension $D$ (see \cite{car:1958}). This also implies that the Hochschild dimension of $\ca$ is $D$ (see \cite{ber:2005}). 
By applying the functor $\Hom_\ca(\, \cdot  , \ca)$ to the chain complex $K(\ca)$ of left $\ca$-modules one obtains the cochain complex $L(\ca)$ of right $\ca$-modules 
$$
L(\ca) : \qquad \quad 0\rightarrow \ca\stackrel{b'}{\rightarrow} \dots\stackrel{b'}{\rightarrow}\ca^!_n \otimes \ca \stackrel{b'}{\rightarrow} \ca^!_{n+1} \otimes \ca \stackrel{b'}{\rightarrow} \cdots 
$$
where $b'$ is the left multiplication by $\sum_k\theta^k\otimes e_k$ in $\ca^!\otimes \ca$ and $(e_k)$ is a basis of $E$ with dual basis $(\theta^k)$. The algebra $\ca$ is said to be \emph{Koszul-Gorenstein} if it is Koszul of finite global dimension $D$ as above and if $H^n(L(\ca))=\IC \, \delta^n_D$. Notice that this implies that $\ca^!_n\simeq \ca^{!\ast}_{D-n}$ as vector spaces (this is a version of Poincar\'e duality).

Finally, a graded algebra $\ca=\oplus_n \ca_n$ is said to have \emph{polynomial growth} whenever there are a positive $C$ and $N \in \IN$ such that
$$
\dim (\ca_n)\leq C n^{N-1} 
$$
for any $n \in \IN$ and a {\sl Koszul Artin-Schelter} algebra of global dimension $D$ is simply a Koszul Gorenstein algebra of global dimension $D$ which has polynomial growth.

As before, let $E$ be a finite-dimensional vector space with the tensor algebra $T(E)$ 
endowed with its natural filtration $F^n(T(E))=\bigoplus_{m\leq n} E^{\otimes^m}$.
A {\sl nonhomogeneous quadratic algebra} \cite{pos:1993} and \cite{bra-gai:1996}, 
is an algebra $\fraca$ of the form
$$
\fraca=A(E,P)=T(E)/(P)
$$
where $P$ is a subspace of $F^2(T(E))$ and where $(P)$ denotes as above the two-sided ideal of $T(E)$
generated by $P$. The filtration of the tensor algebra $T(E)$ induces a filtration $F^n(\fraca)$ of $\fraca$ and one associates to $\fraca$ the graded algebra
$$
\mbox{gr} (\fraca)=\oplus_n F^n(\fraca)/F^{n-1}(\fraca) 
$$
with the convention $F^n(\ca)=0$ for $n<0$.
Let $R$ be the image of $P$ under the canonical projection of $F^2(T(E))$ onto $E\otimes E$ and let $\ca=A(E,R)$ be the homogeneous quadratic algebra $T(E)/(R)$; this $\ca$ is referred to as the {\sl quadratic part} of $\fraca$. There is a canonical surjective homomorphism 
$$
\mbox{can} :\ca\rightarrow \mbox{gr}(\fraca) 
$$
of graded algebras.
One says that $\fraca$ has the {\sl  Poincar\'e-Birkhoff-Witt (PBW) property} whenever this homomorphism is an isomorphism. The terminology  comes from the example 
where $\fraca=U(\fracg)$ the universal enveloping algebra of a Lie algebra $\fracg$.
A central result (see  \cite{bra-gai:1996} and \cite{pol-pos:2005}) states that if $\fraca$ has the PBW property then the following conditions are satisfied: 
\begin{align}\label{i-ii}
& \mathrm{(i)} \quad P\cap F^1 (T(E))=0 , \nn \\
& \mathrm{(ii)} \quad (P \cdot E + E \cdot P)\cap F^2 (T(E)) \subset P 
\end{align}
\noindent 
and that conversely, if the quadratic part $\ca$ is a Koszul algebra, the conditions 
$\mathrm{(i)}$ and $\mathrm{(ii)}$ imply that $\fraca$ has the PBW property.
 Condition (i) means that $P$ is obtained from $R$ by adding to each non-zero element of $R$ terms of degrees 1 and 0. That is there are linear mappings $\psi_1:R\rightarrow E$ and $\psi_0:R
 \rightarrow \IC$ such that one has
\beq
P=\{r+\psi_1(r)+\psi_0(r) \bbbone \,\, \vert \,\, r \in R\}
\label{PR}
\eeq
giving $P$ in terms of $R$. Condition (ii) is a generalisation of the Jacobi identity (see \cite{pol-pos:2005}).


\end{document}